\documentclass[12 pt]{amsart}
\usepackage{amssymb,latexsym,amsmath,amscd,amsthm,graphicx, color}
\usepackage[all]{xy}
\usepackage{pgf,tikz}
\usepackage{mathrsfs}
\usetikzlibrary{arrows}
\definecolor{uuuuuu}{rgb}{0.26666666666666666,0.26666666666666666,0.26666666666666666}
\definecolor{xdxdff}{rgb}{0.49019607843137253,0.49019607843137253,1.}
\definecolor{ffqqqq}{rgb}{1.,0.,0.}

\definecolor{uuuuuu}{rgb}{0.26666666666666666,0.26666666666666666,0.26666666666666666}
\definecolor{qqwuqq}{rgb}{0.,0.39215686274509803,0.}
\definecolor{zzttqq}{rgb}{0.6,0.2,0.}
\definecolor{xdxdff}{rgb}{0.49019607843137253,0.49019607843137253,1.}
\definecolor{qqqqff}{rgb}{0.,0.,1.}
\definecolor{cqcqcq}{rgb}{0.7529411764705882,0.7529411764705882,0.7529411764705882}

\theoremstyle{plain}

\newtheorem{theorem}[subsection]{Theorem}

\newtheorem{lemma}[subsection]{Lemma}

\newtheorem{prop}[subsection]{Proposition}
\theoremstyle{definition}
\newtheorem{defi}[subsection]{Definition}
\newtheorem{cor}[subsection]{Corollary}

\newtheorem{remark}[subsection]{Remark}

\newcommand{\sci}{\subset}
\newcommand{\set}[1]{\{#1\}}

\newcommand{\tit}{\textit}

\newcommand{\B}{\boldsymbol}
\newcommand{\C}[1]{\mathcal{#1}}

\newcommand{\te}{\text}

\begin{document}

To appear,  Dynamical Systems 
\title{Optimal Quantization via Dynamics}

\author{Joseph Rosenblatt}
\address{Department of Mathematical Sciences \\
Indiana University-Purdue University Indianapolis\\
402 N. Blackford Street \\
Indianapolis, IN 46202-3217, USA.}
\email{rosnbltt@illinois.edu}

 \author{Mrinal Kanti Roychowdhury}
\address{School of Mathematical and Statistical Sciences\\
University of Texas Rio Grande Valley\\
1201 West University Drive\\
Edinburg, TX 78539-2999, USA.}
\email{mrinal.roychowdhury@utrgv.edu}

\subjclass[2010]{94A34, 37A45, 11J71, 60B05}
\keywords{Optimal quantizers, quantization error, dynamical systems, Diophantine approximation, and independent random variables}
\date{January 24, 2020}
\thanks{}

\begin{abstract}
Quantization for probability distributions refers broadly to estimating a given probability measure by a discrete probability measure supported by a finite number of points. We consider general geometric approaches to quantization using stationary processes arising in dynamical systems, followed by a discussion of the special cases of stationary processes: random processes and Diophantine processes.  We are interested in how close stationary process can be to giving optimal $n$-means and $n^{th}$ optimal mean distortion errors.  We also consider different ways of measuring the degree of approximation by quantization, and their advantages and disadvantages in these different contexts.
\end{abstract}
\maketitle

\pagestyle{myheadings}\markboth{Joseph Rosenblatt and Mrinal Kanti Roychowdhury}{Optimal Quantization via Dynamics}

\section{Introduction}\label{intro}
\subsection{The General Setting}\label{intosubsecA}

Let $\mathbb R^d$ denote the $d$-dimensional Euclidean space with the Euclidean metric $\|\cdot\|$. Let $P$ be a Borel probability measure on $\mathbb R^d$. Let $V_n =V_n(P)$
be
\[V_n(P)=\inf_{\alpha \in \C D_n} \int \min_{a \in \alpha} \|x-a\| dP(x),\]
where $\C D_n:=\set{\alpha \sci \mathbb R^d : 1\leq \te{card}(\alpha)\leq n}$. We assume that $\int\|x\| dP<\infty$ to make sure that there is a set $\alpha^*_n$ for which the infimum occurs (see \cite{AW, GKL, GL, GL2}). Here $V_n(P)$ is the \tit{optimal mean distortion error}.  The set $\alpha_n^*$ for which we get the optimal mean distortion error is called an \tit{optimal set of $n$-means} and the elements of an optimal set are called \tit{optimal quantizers}. For some mathematical background on quantization, see the text by Graf and Luschgy~\cite{GL}. For some recent work on optimal quantizers, see \cite{CR1, CR2, DR, GL3,  R1, R2, R3, R4, RR1}.   Note that it is traditional to use $\|x - a\|^2$ in the definition above, but we do not do this here because we want to keep as simple as possible the connection between the integral itself and various ways of approximating the integrand.

Given a fixed $\alpha_n \in \C D_n$, we  call $d_n^{\alpha_n}(x) = \min_{a \in \alpha_n} \|x-a\|$
the {\em point distortion error}, and we call the mean of the point distortion error,
$V_n^{\alpha_n} = \int d_n^{\alpha_n}(x) \, dP(x) = \int \min_{a \in \alpha_n} \|x-a\| dP(x)$, the {\em mean distortion error}.
We consider various methods of generating {\em model} $\alpha_n$. We are particularly
interested in estimates for the rate of convergence to zero, as $n$ goes to infinity, of the distortion errors of model sequences that have been generated by some dynamical stochastic process.  If we know both a) the rate of convergence to zero of the distortion errors in the optimal case, and b) in a given model we have good estimates for the rate of convergence to zero of the distortion errors, then we will have some understanding of how far from optimality these models are.

Clearly, if we had a uniform estimate for the rate of convergence to zero of the point distortion error,
then by Lebesgue's Bounded Convergence Theorem  we would obtain a rate estimate for the mean
distortion error.  Indeed, given $(\rho_n)$, even if we just had, for a.e. $x$ with respect to  $P$, an
estimate $\min_{a \in \alpha_n} \|x-a\| = O(\rho_n)$ as $n\to \infty$, we would know  that also $\int \min_{a \in \alpha_n} \|x-a\| dP(x) = O(\rho_n)$ as $n \to \infty$.   However,
estimates for the mean distortion error will generally give us limited insight into the best estimates for the point distortion error that would hold for a.e. $x$.

In the case that the support $K$ of $P$ is compact, there is a geometric approach for getting a uniform estimate for the point distortion error.  Take any $\alpha_n  \in \C D_n$.  Think of it as generated by some process that we believe has a distortion error that is close to the optimal distortion error.  Let $r_n^{\alpha_n}$ be the minimum radius $r$ such that the balls of radius $r$ with centers at $a \in \alpha_n$ cover $K$.   We call $r_n^{\alpha_n}$ the {\em geometric distortion error}.  This distance provides a good geometric proxy for the distortion errors.  Indeed, it is clear that uniformly the point distortion error for $\alpha_n$ is not larger than $r_n^{\alpha_n}$ and hence the minimal mean distortion error $\int \min_{a \in \alpha_n^*} \|x-a\| dP(x)$
is no larger than $r_n^{\alpha_n}$.  We will see that many stochastic processes can give good estimates for $r_n^{\alpha_n}$ with a suitable choice of $\alpha_n$, and hence using these processes not only would we have an estimate for the point distortion error, but we also would have an estimate for the mean distortion error.

\subsection{Using Stochastic Processes}\label{introsubsecB}

Here is the full range of how one might consider distortion errors generated by some sequence of stochastic processes $\mathbf {\mathcal T} =(\mathcal T_k:k\ge 1)$.  We assume that the processes are defined on an underlying state space $(\Omega,\beta_p,p)$ which might be a probability space, but generally would be a $\sigma$-finite measure space.  For $\omega \in \Omega$, the range values of $\mathcal T_k(\omega)$ are in another measure space $(Y,\beta_P,P)$ which is also naturally a metric space with metric $\delta_Y$.  Elements $\omega \in \Omega$ are the parameters and the values $(\mathcal T_k(\omega): k \ge 1)$ the quantizers associated with the parameter $\omega$.

Given $\omega \in \Omega$ and $y \in Y$, we form the point distortion errors $d_n^{\mathbf {\mathcal T}}(\omega,y) = \min\limits_{1 \le k \le n} \delta_Y(y,\mathcal T_k(\omega))$.  The mean distortion error $I_n^{\mathbf {\mathcal T}}(\omega) = \int d_n^{\mathbf {\mathcal T}}(\omega,y) \, dP(y)$ is itself a stochastic process implicitly dependent on the choice of $\omega$.  For a full analysis, we want to calculate the distribution function and  all the moments of $d_n^{\mathbf {\mathcal T}}(\omega,y)$.  So in particular, we would want to know the distributional behavior and moments of $I_n^{\mathbf {\mathcal T}}(\omega)$.   It would be ideal to have the distribution functions and moments of these quantities  explicitly in terms of any additional parameters defining the stochastic process $(\mathcal T_k:k\ge 1)$.  But we may have to accept just  good estimates for sizes of these as a function of $n$.  In the case when the range space $Y$ is actually a Euclidean space $\mathbb R^l$ with the usual Euclidean metric for $\delta$, then we would also want to know how these quantities depend on the dimension $l$.

As indicated above, we can always try to overestimate these quantities using the geometric distortion error.  It turns out that sometimes the best we can do is to overestimate the geometric distortion error using the discrepancy of the stochastic processes $(\mathcal T_k:k\ge 1)$.  This discrepancy is the main tool in Monte-Carlo integration for the stochastic process and it plays an important role in quantization too.    In this case, the geometric distortion would be calculated in terms of $\delta_Y$ and would be a function primarily of $\omega$ and $n$.  The discrepancy would also depend on $\omega$ and $n$, but would be calculated with respect to a test family of sets, for example the $l$-dimensional bounded rectangles in the case that $Y = \mathbb R^l$.

In summary, this is the basic program that we use.  We choose a sequence $\alpha = (\alpha(k): k \ge 1)$ by via some stochastic process that is explicit, even though it might not give the optimal quantizers.  We then compute an estimate for the mean distortion error, the point distortion error, or the geometric distortion error using the $n$ points $\alpha_n = (\alpha(1),\dots,\alpha(n))$.  Say we have done this for the geometric distortion error.  We know that
\[V_n(P) \le \int \min_{1 \le k \le n} \|x-\alpha (k)\| dP(x) \le r_n^{\alpha_n}.\]
If we also have $r_n^{\alpha_n} \le K V_n(P)$, then we have a good estimate for the optimal mean distortion error using the geometric distortion error.  The best outcome would be if in fact
\[V_n(P) \sim  \int \min_{1 \le k \le n} \|x-\alpha (k)\| dP(x) .\]
Then $\alpha_n = (\alpha(1),\dots,\alpha(n))$ would be what we call an {\em asymptotically optimal set of $n$-means}.   An even better situation would be if in addition we had $r_n^{\alpha_n} \sim V_n(P)$.  However, this last estimate, indeed both of these, seem to rarely be the case.

\subsection{Outline}\label{introsubsecC}

The comments above explain why we focus on results for stochastic processes that give estimates of the form $r_n^{\alpha_n} = O(\rho_n)$ with $\rho_n \to 0$ as $n \to \infty$.  We are most interested in explicit sequences that give asymptotically optimal $n$-means, for given probability measures, but these are difficult to obtain.  So we are also interested in methods that would seem likely to give good estimates of this type, and how close the error rate is to that of the asymptotically optimal $n$-means.

In this article, we first consider general dynamical models in Section~\ref{sec2} for which there usually is correlation of the outputs.  We prove a Baire category result to show that there is no global distortion rate for this large class.  Then we discuss what is available for the more special case of random models in Section~\ref{sec3} with uncorrelated variables.  Following this, we look at another special class of dynamical systems, the Diophantine models in Section~\ref{sec4}.  For these there is again correlation, but now number theory plays a critical role.  Each of these approaches has advantages over the other.  They also have advantages over carrying out the detailed, hard work needed to construct explicit optimal $n$-means.  The trade-off is that one generally obtains only at best an estimate that is on the order of the optimal results.

\section{Dynamical models}\label{sec2}

Suppose $\tau$ is an invertible, Lebesgue measure $m$ preserving transformation of $[0,1]$.  Sometimes we may want to consider maps acting on probability spaces with a more complicated geometry, but for many of the points we want to make in this article using $[0,1]$ as the underlying space is perfect for giving the basic results.

Suppose $\tau$ is ergodic.  Then for a.e. $x \in [0,1]$, the orbit $\{\tau^k(x): k \ge 1\}$ is dense in $[0,1]$.  For any $y \in [0,1]$, consider the approximation $d_n^\tau(x,y) = \min\limits_{1\le k \le n} |\tau^k(x)-y|$.  We know that for every $y$ and a.e. $x$, we have $d_n^\tau(x,y) \to  0$ as $n\to \infty$.  We would like to know estimates for the rate that this tends to zero.  Some ways this could be considered are these, with each condition being weaker than the previous one: a) what rate does the $\max\limits_{y \in [0,1]} d^\tau_n(x,y)$ tend to zero for a.e. $x$, b) what rates does $d_n^\tau(x,y)$ tend to zero for every (or just a.e.) $y$ and a.e. $x$, or c) what rate does $\int_0^1 d^\tau_n(x,y)\, dm(y)$ tend to zero for a.e. $x$?   The ergodicity of $\tau$ guarantees that $M_n^\tau = \max_{y\in [0,1]} d_n^\tau(x,y)$ tends to zero for a.e. $x$, and so each of these measures do converge to zero.  We are asking for types of quantization error rate given by the $n$ means $\{\tau(x),\dots,\tau^n(x)\}$.

We do want to keep in mind that there is another quantization error that is stronger than all of the ones above: a gap measurement.  For each $n$, take $(x_k:0 \le k \le n+1)$ to be $(\tau^k(x): 1 \le k \le n)\cup \{0,1\}$ in increasing order i.e. $0 = x_0 \le x_1 \le \dots \le x_n \le x_{n+1} = 1$.  Let the gap measurement be $g_n^\tau (x) = \max \{x_{k+1} - x_k: 0 \le k \le n\}$.  Knowing the rate this goes to zero would be the best we could expect to do since $g_n^\tau (x) \ge 2 \max\limits_{y \in [0,1]} d^\tau_n(x,y)$.  But also, the geometric distortion error is the same thing since $2r_n^\tau(x) = g_n^\tau (x)$.

It is clear that the same question as asked above can be considered for any minimal dynamical system.  That is, take a compact metric space $(X,d_X)$ and a minimal homeomorphism $\tau$ of $X$.  Let $d_n^\tau(x,y) = \min\limits_{1\le k \le n} d_X(\tau^k(x),y)$.  What can be said about how fast this tends to zero for each $y$ or maximized over $y \in X$?

\begin{remark}\label{OtherGeo}
We expect that the questions above for ergodic dynamical systems might be best asked in the original natural geometric structure on which the map $\tau$ is defined.  For example, take an ergodic automorphism $\tau$ of the two torus $\mathbb T^2$.   Replace the distance function of absolute value by the two dimensional Euclidean distance inherited on $\mathbb T^2$ from the natural map $\pi: [0,1]^2 \to \mathbb T^2$.  See Remark~\ref{Auto} for more details.

We could also be considering a differentiable mapping $\tau$ on a manifold $M$ that is ergodic with respect to a natural probability measure on $M$.  These types of geometric structures have been studied extensively from the viewpoint of generalizing the classical discrepancy calculations that we commented on earlier.  But there seems to be not as much known about the quantization approximation itself.
\end{remark}

It is evident also from a  number of cases that a perhaps better measure of the quantization error would be $I_n^\tau = \int_0^1 d_n^\tau(x,y) \, dm(y)$, taken for any $x \in [0,1]$.  Clearly, $I_n^\tau(x) \le M_n^\tau(x)$.
See in particular this issue as discussed in Section~\ref{sec3} and Section~\ref{sec4}.

\subsection{Rates are Always Non-generic}\label{Nongeneric}

The above discussion show us that the following Baire category result gives us some important information about the limits of geometric distortion errors for the class of stationary dynamical systems.    We take $\mathcal M$ to be the group of invertible, measure-preserving transformations of the probability space $([0,1],m)$.  We consider the weak topology on $\mathcal M$ i.e. the strong operator topology on the group of continuous linear operators $T^\tau (f) = f\circ \tau$, with $f\in L^2([0,1])$.  A standard fact is that the weak topology on $\mathcal M$ is given by a complete pseudo-metric.  In particular, it is a {\em Baire space} i.e. any intersection of a countable set  of open, dense sets is also dense.  We say that a set $\mathcal G$ in a Baire space is residual if it contains such a dense $G_\delta$ set, and a set  $\mathcal F$ is meager if it is a complement of a residual set.

\begin{theorem}\label{Notypical} Take a fixed $\delta > 0$ and $\rho_n \to 0$, $\rho_n > 0$ for all $n$.  Consider the set $\mathcal F$ of maps $\tau\in \mathcal M$ such that on a set of $x \in [0,1]$ with measure at least $\delta$, we have $I_n^\tau(x) \le K\rho_n$ for some constant $K$ and for large enough $n$, both depending on $x$.  Then this is a meager set in the weak topology; actually its complement is a dense $G_\delta$ set.
\end{theorem}
\begin{proof} We write $\mathcal F$ as
\[\bigcup\limits_{K=1}^\infty \bigcup\limits_{N=1}^\infty \bigcap\limits_{n=N}^\infty \{\tau: m\{x: I_n^\tau(x) \le K \rho_n\} \ge \delta\}.\]
Let $\mathcal B_N = \bigcap\limits_{n=N}^\infty \{\tau: m\{x: I_n^\tau(x) \le K \rho_n\} \ge \delta\}.$  We can see this is closed in the weak topology determined by $\Delta$ if we show each of the sets $\{\tau: m\{x: I_n^\tau(x) \le K \rho_n\} \ge \delta\}$ is closed in this topology.  We can see this last fact is true as follows.  Take any sequence $(\tau_s)$ in $\{\tau: m\{x: I_n^\tau(x) \le K \rho_n\} \ge \delta\}$.  Suppose $(\tau_s)$ converges to $\sigma \in \mathcal M$ in the metric $\Delta$.  This means that for all $k \ge 1$, $\|f\circ \tau^k_s - f\circ \sigma^k\|_1 \to 0$ as $s\to \infty$ for every $f\in L^1(X)$.  So for all $k\ge 1$ and $x$, $\int |x - \tau_s^k(y)|\, dm(y) \to \int |x - \tau^k(y)|\, dm(y)$ as $s\to \infty$.  Hence, for all $n$ and $x$,
 $I_n^{\tau_s}(x) = \int \min\limits_{1 \le k \le n} |x - \tau_s^k(y)|\, dm(y)$ converges to
 $I_n^\tau (x) =  \int \min\limits_{1 \le k \le n} |x - \tau^k(y)|\, dm(y)$ as $s\to \infty$.  Thus,
 because $m\{x: I_n^{\tau_s}(x) \le K \rho_n\} \ge \delta$ for all $s$, we also have
 $m\{x: I_n^\tau (x) \le K \rho_n\} \ge \delta$.

This shows that $\mathcal F$ is an $F_\sigma$ set in the weak topology.  So to prove our result, we need only to show that $\mathcal B_N$ contains no interior.  Suppose on the contrary that $\tau \in W \subset B_N$ with $W$ being a weak open set.  Then by standard methods one can show there is a cyclic transformation $\tau_0 \in W$.  Even more, we can have a partition $\{E_j: j = 1,\dots,J\}$ of $X$ with each set $m(E_j) = 1/J$ and such that if $\tau_0$ is any measure preserving map that sends the sets $E_j$, sending $E_j$ to $E_{j+1}$ for all $j$, and $E_{J}$ to $E_1$, then $\tau_0 \in W$.  Now (if necessary) we modify $\tau_0$ (but keeping it a cyclic permutation of the sets $E_j$ as above).  First, take intervals $S_j$ of the same small, positive length and let $D_j = E_j\backslash S_j$ and $C_j = E_j\cap D_j$.  We now need to also adjust the lengths of $S_j$, keeping them of some small, positive length (possible different) so that the sets $D_j$ all have the same measure.  We now modify $\tau_0$ so that it actually cyclically permutes the sets $D_j$ in the sense above, and as a result the sets $C_j$ too, with the same permutation of the index $j \to j+1, 1 \le j < J$ and $J \to 1$.  These adjustments are made so that $\tau_0$ is still in $W$.  Considering the center half of the intervals $S_j$, we see that for all $x \in \bigcup\limits_{j=1}^J D_j$, we have for all $m$,  $I_m^{\tau_0}(x) \ge \gamma$ for some, possibly quite small, value $\gamma > 0$.  Now, as part of these adjustments, we can arrange that $\bigcup\limits_{j=1}^J D_j$ has measure at least $1 - \frac {\delta}2$.   This combination of estimates guarantees that for all $m$, there exists $x$ such that $\gamma \le I_m^{\tau_0}(x) \le K\rho_m$.
But since $\rho_m \to 0$ as $m\to \infty$, we cannot have $\gamma \le  K\rho_m$ for all $m$.   The conclusion is that $B_N$ cannot have interior in the weak topology.
\end{proof}

\begin{remark}  This result means that there is no rate for $I_n^\tau$, no matter how slow, for the quantization error rate which would apply to a Baire category large set of maps $\mathcal M$.  Since the ergodic mapping are themselves a dense $G_\delta$ set, for any specific quantization error rate as above, there would be (many) ergodic mappings that do not satisfy this condition.  This means the same thing would be true when attempting to get a general rate result for $d_n^\tau(x)$.   Moreover, it means that there cannot be a pointwise rate either i.e. some rate $(\rho_n)$ such that for all $\tau \in \mathcal M$, $I_n^\tau(x) \le \rho_n$ for large enough $n$ depending on $x$, which holds for all $x$ in some set of non-zero measure.
\end{remark}

\begin{remark}  There has recently been a similar result proved by A. Junqueira~\cite{J} in the topological setting.  It seems likely that just assuming a one point rate as in \cite{J} will not be adequate in the measure-theoretic category because maps can be changed on a null set without changing their measure-theoretic dynamics.
\end{remark}

We state this corollary of Theorem~\ref{Notypical} to emphasize what has been proved.

\begin{cor}\label{nogap}  Take a fixed $\delta > 0$ and $\rho_n \to 0, \rho_n > 0$ for all $n$.  Consider the set $\mathcal F$ of maps $\tau \in \mathcal M$ such that on a set of $x\in [0,1]$ with measure at least $\delta$, we have the maximum gap $g_n^\tau(x) \le K \rho_n$ for some constant $K$ and for all large enough $n$, depending on $x$.  Then this is a meager set in the weak topology; actually its complement is a dense $G_\delta$ set.
\end{cor}

This is not to say that there cannot be a good rate result for a large  set of ergodic maps.  Indeed, take any ergodic rotation $\tau_\alpha$ given by $\alpha$ whose terms in the simple continued fraction for $\alpha$ are bounded.  Then the result in Graham and Van Lint~\cite{GVL}, and also in Schoissengeier~\cite{Sch} shows that $\rho_n = \frac 1n$ does work for the quantization error rate.

It is not clear how pervasive this rate result can be since conjugating these examples does not necessarily preserve the gap structure.  However, we can certainly get some rate for a dense class.  Just take a countable dense set of ergodic mappings and use diagonalization to obtain a sufficiently slow rate.  This gives,

\begin{prop}  There is a countable set $\mathcal S$ of ergodic mappings in $\mathcal M$ which is dense in weak topology and a decreasing sequence $\rho_n \to 0$ as $n\to \infty$ such that $r_n^\tau(x) = O(\rho_n)$ for all $\tau \in \mathcal S$ and a.e. $x$.
\end{prop}

\begin{remark}  It would be very useful to know that via cutting and stacking one can construct an ergodic, rank one map $\tau$ such that $2nd_n^\tau (x)\to 1$ as $n\to \infty$ for a.e. $x$.  It is not clear what rate works for maps like Chacon's map, and it is not likely to have this property.  But another such construction, with many divisions and some spacers at each inductive step should give at least an ergodic mapping, and one with this asymptotically optimal point distortion rate.
\end{remark}

\subsection{Discrepancy}\label{Discrepancy}
The quantization process is closely related to discrepancy estimates. See Kuipers and Niederreiter~\cite{KN}, especially the chapter notes, for a wealth of background information and references on discrepancy.  Another more recent reference that is also excellent is Drmota and Tichy~\cite{DT}.  We again take our interval modulo one, but we suppress this in the notation for simplicity.

\begin{defi} Given a sequence $ \B{\alpha} = (\alpha(k):k\ge 1)$ in $[0,1]$, the discrepancy \[D_n(\B{\alpha}) =
\sup\{|\frac 1n\sum\limits_{k=1}^n 1_{[x,y]}(\alpha(k)) - (y - x)|: 0 \le x < y \le 1\}.\]
The discrepancy
\[D_n^*(\B{\alpha}) =
\sup\{|\frac 1n\sum\limits_{k=1}^n 1_{[0,y]}(\alpha(k)) - y|: 0 <  y \le 1\}.\]
\end{defi}

\noindent It is easy to see that $D_n^* \le D_n \le 2D_n^*$.

Now if $D_n < \delta$, then for any interval $I$ of length $\delta$, there must be some $\alpha_k \in I$ with $k \le n$.  So $\min\limits_{1 \le k \le n} |x - \alpha(k)| \le \delta$.  Hence, we have the following useful basic estimate:

\begin{lemma}\label{gapvsD} For any $x \in [0,1]$, we have $\min\limits_{1 \le k \le n} |x - \alpha(k)| \le D_n(\alpha)$.
\end{lemma}

In particular, we can consider the discrepancy $D_n(O_\tau(x))$ where $O_\tau(x)$ is the forward time sequence $(\tau^k(x): k \ge 1)$ for an erogdic map $\tau$.  As a consequence of the estimate in Lemma~\ref{gapvsD} and the non-generic behavior in Theorem~\ref{Notypical},
we have this Baire category result.

\begin{cor}\label{NotypicalDiscrepancy}  Take a fixed $\delta > 0$ and $\rho_n \to 0$, $\rho_n > 0$ for all $n$.  Consider the set $\mathcal F$ of maps $\tau\in \mathcal M$ such that on a set of $x \in [0,1]$ with measure at least $\delta$, we have $D_n(O_\tau(x)) \le K\rho_n$ for some constant $K$ and for large enough $n$, both depending on $x$.  Then this is a meager set in the weak topology on $\mathcal M$.
\end{cor}

\begin{remark}\label{Auto}
Here is an example of related dynamical rate questions that occur when the underlying measure space is changed for which the only result on geometric distortion error rates that we know is through discrepancy estimates. Take an ergodic automorphism $A$ of the two torus $\mathbb T^2$.  Let $d_n^A(x,y)= \min\limits_{1 \le k \le n} d_{\mathbb T^2}(y,A^kx)$, where $d_{\mathbb T^2}$  is the natural Euclidean metric distance (taken mod one actually).
Let $I_n^A(x) = \int\limits_{\mathbb T^2} d_n^A(x,y) \, d\lambda_{\mathbb T^2}(y)$.  What can be said about how fast this goes to zero for a.e. $x$?  The estimate to prove or deny here would be that $I_n^A = O(1/n)$ since the ergodic map $A$ is Bernoulli, and hence reasonably quickly mixing.
  If instead, we consider $d_n^A(x) = \max\limits_{y\in \mathbb T^2} d_n^A(x,y)$, what is the rate that this overestimate for $I_n^A$ goes to zero for a.e. $x$?
What we are most interested in is if these rates are really faster than what would come out of using overestimates given by the natural two dimensional discrepancy
\[D_n^A = \sup_{r<s,u<v}  \left |\frac 1n\sum\limits_{k=1}^n 1_{[r,s]\times [u,v]}(T^kx)
- (s-r)(v-u)\right |.\]
See Losert, Nowak, and Tichy~\cite{LNT}, and Nowak and Tichy~\cite{NT}, and citations in these articles, for the overestimates on discrepancy that are available here, and in higher dimensions.  At least when one eigenvalue of $A$ has modulus larger than $1$, we have $D_n^A = O(\ln^5(n)/\sqrt n)$.
\end{remark}

\subsection{Shrinking Targets}\label{Shrink}

Consider an ergodic mapping of $[0,1]$ and a fixed sequence $(\rho_n)$ decreasing to zero.  A {\em geometric shrinking target} result would give information about the Lebesgue measure of the set of points $x$ such that we have for all $y$ (or perhaps only a.e. $y$), infinitely often $\tau^n(x) \in [y-\rho_n,y+\rho_n]$.     Typically we would want to know that either this holds for a.e. $x$ , or in contrast to this it might hold only on a null set of $x$.  The first case is called a {\em visible geometric shrinking target}, and the second case is called an {\em invisible geometric shrinking target}.  These extremes are the best types of geometric shrinking target properties.

It turns out that the geometric distortion error $r_n^\tau(x)$ gives an estimate for a shrinking target rate.  Indeed, using Boshernitzan's Theorem ~\cite{B}, we prove this theorem in \cite{RR2}.

\begin{prop}\label{Shrinkrate} Suppose we have a sequence $(\rho_n)$ such that for a.e. $x$,  we have the geometric distortion error $r_n^\tau(x) \le \rho_n$ for large enough $n$ depending on $x$.  Then we have the geometric shrinking target behavior that  for a.e. $x$, we have for all $y$, infinitely often $\tau^n(x) \in [y-\rho_n,y+\rho_n]$.
\end{prop}

It follows that there is always some shrinking target rate.

\begin{prop}  For every ergodic mapping $\tau$, there is a sequence $(\rho_n)$ tending to zero so that for all $y\in [0,1]$, the sequence of intervals $(B_{\rho_n}(y):n \ge 1)$ is an a.e. visible geometric shrinking target with respect to $\tau$.
\end{prop}

\begin{remark}\label{limitations} Shrinking target theorems for general dynamical systems are studied extensively in the Rosenblatt and Roychowdhury~\cite{RR2}, where there also are background and references on this topic.  In the end, the best values of $r_n^\tau$ are not known in general.  Indeed, we see in this article sometimes we have to use the (generally) much larger discrepancy of the sequence $(\tau^k(x): k\ge 1)$ to get an upper bound on the geometric distortion error.  Similarly, there is generally a loss of speed when using the geometric distortion error to derive a geometric shrinking target results.  This is what allows for difference in Baire category results when considering suitable rates in shrinking target theorems as opposed to geometric distortion error rates.   In particular, in \cite{RR2}, it is shown there that shrinking target properties are sufficiently flexible so the behavior is actually generic in many non-trivial cases, unlike with quantization error rates where the behavior is only for a first category class of maps, as noted in Theorem~\ref{Notypical} and Corollary~\ref{NotypicalDiscrepancy}.
\end{remark}

\section {IID Models}\label{sec3}

Consider a method of randomly generating $n$-means for uniform measure on the interval $[0,1]$ modulo one.  We take $\B{\alpha} = (\alpha(k): k \ge 1)$ to be IID random variables with uniform distribution.  We actually are taking $\alpha(k,\omega)$ with $\omega \in \Omega$ as the model underlying probability space $(\Omega,P)$, but we will suppress the dependence on $\omega$ if it will not create confusion.

Suppose we want an estimate for $d_n(\{\alpha_1,\dots,\alpha_n\})$.
The simplest approach would be to estimate how many terms $(\alpha(1),\dots,\alpha(n))$ are needed so that each interval $I_j = [j/M,(j+1)/M], j = 0,\dots,M-1$ contains at least one point, with high probability.  This will guarantee that our quantization error $\int_0^1 \min\limits_{1 \le k \le n}|x - \alpha(k)| \, dx$ is no larger than $M\int_0^{1/M} x \,dx = O(1/M)$.

\begin{prop}\label{IIDeasy} With probability $1 - 1/\ln (n)$,
the point distortion error $d_n(\{\alpha_1,\dots,\alpha_n\})= O(\ln(n)/n)$.
\end{prop}
\begin{proof} It is easiest to consider the probability of the complementary case: no term $\alpha(k),k=1,\dots,n$ is in some $I_j$.  This probability is $(1 - \frac 1M)^n$ for each such $j$.  So an estimate for the entire scope of the possibility is $M(1 - \frac 1M)^n$.  Taking $M = n/\ln(n)$ as a real variable would give for large $n$,
$M(1 - \frac 1M)^ n \sim 1/\ln(n)$.  Hence, with probability $1 - 1/\ln (n)$, each $I_j$ contains some $\alpha(k), 1 \le k \le n$.  This gives the estimate $1/M = \ln (n)/n$ for the quantization error with this probability.
\end{proof}

\begin{remark}\label{IIDimprove}  Proposition~\ref{IIDeasy} only gives convergence in measure as $n$ goes to $\infty$, but a simple increase in the size of $M$ can guarantee an almost sure result.  Note: instead of the optimal distortion error of $C/n$, this approach is giving a somewhat worse estimate of $C\ln (n)/n$.
\end{remark}

We can get more information from a distributional calculation. Consider the probability $P(\{\omega:n\min\limits_{1\le k \le n} |x - \alpha(k,\omega)| \ge t\})$.  It is easy to see that this is $(1 - \frac {2t}n)^n$.  So scaling of the distortion error by $n$ results in convergence in distribution to the distribution function $d(t) = 1 - e^{-2t}, t \ge 0$, one can also compute expectations, and other moments.  For example, we have the following result.

\begin{prop}\label{IIDdist}
The probabilistic mean of the point distortion error in the IID case satisfies
\[\int_{\Omega} \min\limits_{1\le k \le n} |x - \alpha(k,\omega)|\, dP(\omega) = O(1/n).\]
\end{prop}
\begin{proof}  One computes
\begin{align*}
&\int_{\Omega} n\min\limits_{1\le k \le n} |x - \alpha(k,\omega)|\, dP(\omega)\\
& =\int_0^\infty P(\{\omega: n\min\limits_{1\le k \le n} |x - \alpha(k,\omega)| \ge t\}) \, dt
= \int\limits_0^{n/2} (1  - 2t/n)^n \, dt
= \frac n{2(n+1)}.
\end{align*}
\end{proof}

\begin{remark} Going further than this distributional convergence is not going to be possible because of the Hewitt-Savage Theorem~\cite{HS}.  It shows that if this sequence converges a.e. or even just in measure, then the limit function would be a constant.  The distributional convergence shows that this is not possible.
\end{remark}

Suppose we want to improve our estimate on the point distortion error.  As remarked in Section~\ref{Discrepancy}, the quantization process is closely related to discrepancy estimates. Here is how this works out in the case of IID sequences.  We use the following result of K-L Chung~\cite{KLC} which is the best possible discrepancy result for IID sequences.

\begin{theorem} For a.e. $\omega$,
\[\limsup\limits_{n\to \infty} \frac
{\sqrt {2n}D_n^*(\B{\alpha}(\omega))}{\sqrt {\ln\ln (n)}} = 1.\]
\end{theorem}

\noindent Combining this with Lemma~\ref{gapvsD} gives this upper bound on the distortion error.

\begin{cor} For a.e. $\omega$, $\min\limits_{1 \le k \le n} |x - \alpha(k)(\omega)| =  O(\sqrt {\ln\ln (n)}/\sqrt n)$.
\end{cor}

However, the actual point distortion error rate in the IID case is much faster than what this discrepancy estimate gives.   For simplicity of notation, for the maximum gap measure in this case, write  $g_n(\omega)$ instead of $g_n^\alpha$ where $\alpha = \{\alpha(1,\omega),\dots,\alpha(n,\omega)\}$.  Here is the result of Levy~\cite{L}.

\begin{prop}\label{Levy} For an IID uniform sequence in $[0,1]$, $ng_n(\omega) = O(\ln(n))$.
\end{prop}

What we are considering here is the first step in the classical problem of non-parametric statistics: gap measurements.  Beyond the largest gaps, one can also look for the next largest gap, the next, and so on, and the sizes and distribution of the values of these successive gaps.
This type of order statistics of uniformly distributed IID processes in $[0,1]$ is now very well understood.   For improvement on Levy's Theorem and more on order statistics for uniform random variables see the articles by L. Devroye~\cite{LD1,LD2} and P. Deheuvels~\cite{PD1,PD2,PD3}.  Indeed, Levy's Theorem can be made more specific. For example, in L. Devroye~\cite{LD2}, the following result is shown.

\begin{prop} Almost surely in $\omega$
 \[\liminf\limits_{n\to\infty} (ng_n(\omega) - \ln (n) + \ln\ln\ln (n)) = -\ln(2)\]
and
\[\limsup\limits_{n\to \infty} (ng_n(\omega) - \ln (n))/2 \ln\ln (n) = 1.\]
\end{prop}

There is one important improvement that can be made; see Cohort~\cite{PC}.   If we also integrate with respect to $x$ as above, then there is a.s. convergence to a computable constant.  That is, we switch from a point distortion error to the mean distortion error for the IID model.  Then we have the following.

\begin{prop}\label{cohort} There is a non-zero constant $C$ such that for a.e. $\omega$, \[\int_0^1 n\min\limits_{1\le k \le n} |x - \alpha(k,\omega)| \, dx\]
converges to $C$ as $n\to \infty$.
\end{prop}

\begin{remark} Proposition~\ref{cohort} is proved by calculating variances and using an infinite series method.  Cohort~\cite{PC} actually carries these calculations out in greater generality than our one dimensional setting.  This article contains other interesting results related to a.s. convergence of the random proxy for optimal $n$-means and conclusions that follow about the asymptotic optimality of the random $n$-means.
\end{remark}

\begin{remark} a)  The most important point that the result in Cohort~\cite{PC} gives us is that the mean distortion error is generally much smaller than the point distortion error, or a geometric distortion error like the maximum gap measurement.  These in turn are much smaller than what would be given by using just the discrepancy of the random sequence.  But then it is well known that $\limsup\limits_{n\to \infty} nD_n = \infty$ for every sequence.
So it is not surprising that the same thing is true in this IID model for $nr_n(\omega)$, although we will see in Section~\ref{sec4} that this type of divergence does not hold in some interesting deterministic settings.   In any case, the probability and statistics literature give good results for the IID model for the exact rate at worst that $n r
_n(\omega)$ tends to $\infty$.  Calculating this rate for higher dimensional cubes, and with respect to other common distributions for the IID sequence besides the uniform distribution, would be very worthwhile.
\medskip

\noindent b) By the results in Section~\ref{Shrink}, specifically Proposition~\ref{Shrinkrate}, by the above results, we would get a geometric shrinking target rate of $\ln (n)/n$ for the random sequence of centers.  However, there is a much better result available as given in Shepp~\cite{Sh}.  See Rosenblatt and Roychowdhury~\cite{RR2} for more details.
\end{remark}

If the measure $P$ that we are quantizing is not uniform, then to get a good quantization, we need to adjust the placement of the random variables $(\alpha(k): k \ge 1)$.  The obvious approach is to just take $\alpha(k)$ to be IID with distribution given by the fixed probability measure $P$.  Then we would have the empirical measures $\frac 1n \sum\limits_{k=1}^n \delta_{\alpha(k)}$ converging weakly to $P$.  The result of Theorem 7.5 in Graf and Luschgy~\cite{GL} shows that our random
empirical measure would not be asymptotically  optimal except in the case of uniform measure.  See also the discussion following Theorem 7.5 in \cite {GL} .

\section{Diophantine Models}\label{sec4}

\subsection{The Weyl sequences}

Now consider a method of generating good $n$-means that relies on some classical number theory: a Diophantine model.  We take $\alpha(k,\theta) = \{k\theta\}$ for all  $k \ge 1$.   Here $\theta$ is some irrational number and $\{t\}$ denotes the fraction in $[0,1)$ such that $t = \{t\} + k$ for some integer $k$.

We know that $\B{\alpha}(\theta) = (\alpha (k,\theta): k \ge 1)$ is uniformly distributed on $[0,1]$ and moreover there is a classical estimate in Khinchin~\cite{AK} for the discrepancy $D_n(\B{\alpha}(\theta))$ that holds for a.e. $\theta$.  This estimate come from metric facts about continued fractions and  Diophantine approximation.
This discrepancy estimate is in some sense parallel to the iterated logarithm method of Chung.  See Kuipers and Niederreiter~\cite{KN} for some discussion of this theorem of Khinchin.

\begin{theorem}\label{Khinchin} For any non-decreasing $g$ such that $\sum\limits_{n=1}^\infty \frac 1{g(n)} < \infty$, for a.e. $\theta$, one has for the sequence ${\B{\alpha}(\theta)} = (k\theta \mod 1: k \ge 1)$
	\[nD_n(\B{\alpha}(\theta)) = O(\ln (n) g(\ln \ln (n))).\]
\end{theorem}

\begin{remark} Curiously enough, the discrepancy estimate here is better than the discrepancy estimate available in the IID case.
\end{remark}

From Theorem~\ref{Khinchin} and Lemma~\ref{gapvsD}, we clearly have this parallel to Proposition~\ref{IIDeasy}.

\begin{prop}\label{Weyleasy} If $\delta > 0$, for a.e. $\theta$,
the point distortion error $d_n(\B{\alpha}(\theta)) = O(ln^{(1 +\delta)}(n)/n)$.
\end{prop}

\begin{remark} The estimate in Proposition~\ref{Weyleasy} is not as good as the optimal one that would be $C/n$.  Actually, this estimate is not even as good as the one from the IID model, specifically Proposition~\ref{IIDeasy}.  However, this Diophantine approach has the virtue of being somewhat more explicit than the IID approach.
\end{remark}

The overestimate above from Khinchin's Discrepancy Theorem is most likely too large to give a good rate for the distortion error in the Diophantine model.   For example, see the results in Graham and Van Lint~\cite{GVL}.  Their results on gaps tells us some very interesting facts.  Let $d_n^\theta(y) = \min\limits_{1\le k \le n} |y - \{k\theta\}|$ and let $M_n^\theta =
\max\limits_{y \in [0,1]} d_n^\theta(y)$.

\begin{prop} There is an absolute constant $\gamma > 0$ such that for any irrational $\theta$, we have
\[\liminf\limits_{n\to \infty} n M_n^\theta + \gamma \le \limsup\limits_{n\to \infty} n M_n^\theta.\]
\end{prop}

This means that the quantization error rate obtained from $(\{k\theta\}: k \ge 1)$ can never be optimal in the strict sense, but only optimal by a bounded proportion of the optimal error rate.  These examples also demonstrate another sure way that discrepancy is too large.
But interestingly enough, if $\theta$ has bounded terms in its simple, continued fraction expansion, then $\sup\limits_{n\ge 1} n M_n^\theta < \infty$.  Indeed, this is characteristic of such irrational numbers.

\begin{remark}  There does not seem to be a result in the literature of the following type: there is a sequence $(h_n^*: n \ge 1)$ with $h_n^*$ increasing to $\infty$, such that for Lebesgue a.e. $\theta$, we have $n M_n^\theta = O(h_n^*)$, and for any other sequence $(h_n)$ with this property, $\sup\limits_{n\ge 1} h_n^*/h_n$ is bounded.  Note: the discrepancy estimate we have used shows that for any $\delta > 0$, $n M_n^\theta =O(\ln^{1+\delta}(n))$.  So $\ln^{1+\delta}(n)$ a candidate for $h_n^*$.   However, we do not believe that this rate is optimal and we are still seeking the optimal quantization error rate $h_n^*$ for this model.
\end{remark}

Generally, we have for some $C > 0$,  $nD_n \ge C\ln (n)$ infinitely often; see Kuipers and Niederreiter~\cite{KN}, Theorem 2.2 with $k=1$, while $nd_n^\theta$ is bounded when $\theta$ has bounded terms in its continued fraction expansion.  Note: for these same $\theta$, the overestimate for $nD_n$ is $C \ln (n)$; see Kuipers and Niederreiter~\cite{KN}, Theorem 3.4.

\begin{remark} It is possible that the Diophantine results here can be improved by a couple of other different approaches. One approach would be to take a specific very good value of $\theta$, actually the Golden Mean ratio.  See Motta, Shipman, and Springer~\cite{MSS} where optimal transitivity is studied to limit the gaps in the sequence.  Another approach would be to use bounded remainder sets so that the discrepancy error can be perhaps better controlled.  See both Haynes, Kelly, and Koivusalo~\cite{HK1} and Haynes and Koivusalo~\cite{HK2}.
\end{remark}

\begin{remark}
Another interesting direction to pursue here is to take as our quantization sequences such as $\{2^k\theta\}$, or more generally $\{n_k\theta\}$ for some increasing sequence $(n_k)$.  We do not see right now any clear way to get gap measurements for a point distortion error rate in these cases.  However, there is considerable literature on discrepancy in these cases.  For example, see Philipp~\cite{P} for the case of powers of $2$, that in some ways can be considered as a Diophantine version of the IID model, as can be the case using any sequence $(n_k)$ which is lacunary.  This analogy might suggest that the point distortion error and the mean distortion error are like the ones for the IID uniformly distributed case.  For example, is it true that the maximal gap $g_n$ in $(\{2^k\theta\}: k \ge 1)$ is on the order of $\ln (n)/n$?   Also, see articles by Berkes, Fukuyama, and Nishimura~\cite{BFN}, Aistleitner and Fukuyama~\cite{AF}, and Aistleitner and Larcher~\cite{AL1, AL2} where very interesting results on what types of discrepancy estimates are available for the sequences $(\{n_k\theta\}: k \ge 1)$.  They show that any rate (up to standard restrictions) is possible for a suitable sequence $(n_k)$.
\end{remark}

\subsection{The Farey points}

One additional parallel here is that one can replace the Diophantine model by (in some sense) its close relative: Farey fractions.   It is more natural here to consider each set $\mathcal F_n$ of Farey fractions as a potential good quantization for uniform measure.  The Farey fractions $\mathcal F_n$ of order $n$ are all rational fractions $p/q$ with $q \le n$ and $gcd(p,q) = 1$. The cardinality of $\mathcal F_n$ is $\frac {3n^2}{\pi^2} + (n\ln (n))$.  As in Kargaev and Zhigljavsky~\cite{KZ1}, let $d_n^{\mathcal F} (x) = \min\limits_{p/q \in \mathcal F_n} |x - \frac pq|$.  They give the following mean and point distortion asymptotics.

\begin{prop}
Moreover, the mean distortion error
\[I_n^{\mathcal F} = \int_0^1 d_n^{\mathcal F}(x) \, dx = \frac {3 \ln(n)}{\pi^2 n^2}+ O(\frac 1{n^2})\]
as $n \to \infty$.
\end{prop}

\begin{prop}
For any $\epsilon > 0$ and for a.e. $x$:

\[\lim\limits_{n \to \infty} n^2d_n^{\mathcal F} \ln^{1+\epsilon} (n) = \infty\]
\[\liminf\limits_{n\to \infty} n^2 d_n^{\mathcal F} \ln (n) = 0\]
\[\limsup\limits_{n\to \infty} n^2 d_n^{\mathcal F}/\ln(n) = \infty\]
\[\lim\limits_{n\to \infty} n^2d_n^{\mathcal F}/\ln^{1 + \epsilon} (n) = 0.\]
\end{prop}

\begin{remark} These very good estimates suggest that in the Diophantine case we might have similar results.  One reason for expecting this is that the arguments for the approximation by Farey fractions often take advantage of properties of Farey fractions that also give insights into the Three Gaps Theorem for the Weyl sequence in the Diophantine model, see Polanco, Schultz, and Zaharescu~\cite{PSZ}.  For a recent article about the Three Gap Theorem itself, see T. Van Ravenstein~\cite{VR}.  What is thus suggested is that we can perhaps expect the gaps in $(\{n\theta\})$ to be on the order of $\ln(n)/n$ for a.e. $\theta$.  But the precise over and under-estimates in terms of limit supremum and limit infimum may cause some variation in this.  But also, perhaps for a.e. $\theta$, there is some constant $C(\theta)$ such that
\[\int_0^1 \min\limits_{1 \le k \le n}|x - \{k\theta\}|\, dx \sim \frac {C\ln (n)}n\]
as $n\to \infty$.
\end{remark}

\begin{remark}  We should observe that there is a classical discrepancy estimate when using the Farey fractions $\mathcal F_n$.   The first result in this direction appears in Niederreiter~\cite{N1}; he shows that using the $N_n$ points in $\mathcal F_n$, one has for the discrepancy $\frac {c_1}{\sqrt N_n} \le D_{N_n}^{\mathcal F} \le \frac {c_2}{\sqrt {N_n}}$ for some constants $c_1$ and $c_2$.  Following this, Dress~\cite{FD} gave the very nice explicit value: $D_{N_n}^{\mathcal F} = \frac 1n$.  Given that $N_n \sim \frac {3n^2}{\pi^2}$, this explains the result that Niederreiter had obtained earlier on the discrepancy of the Farey points.
\end{remark}

The asymptotic deviation (in a suitable measure) of the Farey fractions from uniformly placed points is equivalent to the Riemann Hypothesis.  The classical result is this.  Take the Farey fractions $\mathcal F_n$ and let $N_n$ be the cardinality of $\mathcal F_n$.  Let $(f^n_k: 1 \le k \le N_n)$ be $\mathcal F_n$ written in increasing order.  Consider $\Delta_n(s) = \sum\limits_{k=1}^n |x_k - \frac kn|^s$ for some fixed $s$.  Franel~\cite{F} proved that the Riemann Hypothesis is equivalent to $\Delta_n(2) = O(n^r)$ for all $r > -1$.  In particular, this would show that $\Delta_n(2) = O(1/\sqrt n)$, and hence this sum deviation tends to zero as $n \to \infty$.  Also, Landau~\cite{Landau} showed that the Riemann Hypothesis is equivalent to $\Delta_n(1) = O(n^r)$ for all $r > 1/2$.

This way of measuring the offset from a regular distribution represented by uniformly spaced points suggests another way in which we can test our sequence $\B{\alpha}$ for the distance from optimal $n$-means.  We take any sequence $\B{\alpha} = (\alpha(k):k\ge 1)$ generated by random models in Section~\ref{sec2}, Diophantine models in this section, or dynamical methods as in Section~\ref{sec4}.  We take $x_1,\dots,x_n$ to be $\alpha(k),k=1,\dots,n$ written in increasing order.  Then consider the  difference $\Delta_n^{\B {\alpha}}(s) = \sum\limits_{k=1}^n |x_k - \frac kn|^s$ for some fixed $s$, e.g. $s =1$ or $ s=2$. What is the best over-estimate $\rho(n)$ in each of these cases so that $\Delta_n^{\B{\alpha}} (s)= O(\rho(n))$ for all $n \ge 1$?  Indeed, when do we have $\Delta_n^{\B{\alpha}}(s) \to 0$ as $n\to \infty$? This is at least conjecturally the case for the Farey fraction estimate in place of the Weyl sequence. What happens in the Diophantine model?  What happens with a random model or a dynamical systems model?  We do not know the answers here yet, but something can be said in the Diophantine case.  Given $\theta$, take $(x_k^{(\theta,n)}: 1 \le k \le n)$ to be $\{\{k\theta\}: 1 \le k \le n\}$ written in increasing order.  Let $\Delta_n^\theta(s) = \sum\limits_{k=1}^{N_n} |x_k^{(\theta,N_n)} - f_n^k|^s$.  What is the best rate control for this in terms of $n$, given fixed $\theta$?  It is not hard to see using the Three Gaps Theorem that in this case $\Delta_n^\theta (s) = O(n)$.  So to get smaller sizes for $\Delta_n^{\mathcal M}(s)$, we will have to switch to a random model or dynamical model $\mathcal M$ that is not as rigid as the Diophantine model.

\begin{remark}\label{Delta}  Given two sequences $\alpha_1$ and $\alpha_2$, it is not clear what geometric or measure-theoretic property is equivalent to a rate control on $\Delta_n^{(\alpha_1,\alpha_2)}(s) = \sum\limits_{k=1}^n |\alpha_1(k) - \alpha_2(k)|^s$.
\end{remark}

\begin{remark}  It is not too difficult to prove a Baire category result of this type.  Consider a Lebesgue measure-preserving, ergodic invertible map $\tau$ of $[0,1]$.  Given $y\in [0,1]$, let $(x_k^{(\tau,y)}:1\le k\le n)$ be $(\tau^k(y):1 \le k\le n)$ in increasing order.  For $p, 0 < p < \infty$, let the deviation $\Delta_n^{(\tau,y)}(s) =\sum\limits_{k=1}^n |x_k^{(\tau,y)} - \frac kn|^s$.  Given any rate $\rho_n > 0$ such that $\rho_n /n \to 0$ as $n\to \infty$, and any $\epsilon > 0$, the set $\mathcal B$ consisting of mappings $\tau$ such that $\Delta_n^{(\tau,y)}(s) = O(\rho_n)$ for all $y \in B$ with $m(B) \ge \epsilon$, is first category in the ergodic mappings with the usual symmetric pseudo-metric.  Hence, the generic mapping yields no rate result at all.  It is not actually clear at this time if there are ergodic mappings for which there are rate results.  The obvious candidate would be a rotation by an angle whose continued fraction expansion has bounded elements, like the Golden Mean ratio.
\end{remark}

\end{document}